\documentclass[11pt,letterpaper]{article}

\def\sharemode{0}

\ifnum\sharemode=1
\usepackage[15-16]{pagesel}
\fi

\newcommand{\Ent}{\mathrm{Ent}}

\newcommand{\Spec}{\mathrm{Spec}}
\renewcommand{\Re}{\mathrm{Re}\,}
\renewcommand{\Im}{\mathrm{Im}\,}
\newcommand{\ccF}{\mathcal{F}\!}
\newcommand{\primo}{\mathsf{P}^*}
\newcommand{\duo}{\mathsf{D}^*}

\def\stocmode{0}

\def\jamesmode{0}
\def\arxivmode{1}
\def\fastmode{0}

\def\showauthornotes{2}

\def\showkeys{0}
\def\showdraftbox{1}
\def\showcolorlinks{1}
\def\usemicrotype{1}
\def\showfixme{1}

\ifnum\fastmode=0
\usepackage{etex}
\fi

\ifnum\fastmode=0
\usepackage[l2tabu, orthodox]{nag}
\fi

\usepackage{xspace,enumerate}

\usepackage[dvipsnames]{xcolor}

\ifnum\fastmode=0
\usepackage[T1]{fontenc}
\usepackage[full]{textcomp}
\fi

\usepackage[american]{babel}

\usepackage{mathtools}

\usepackage{amsthm}

\newtheorem{theorem}{Theorem}[section]
\newtheorem*{theorem*}{Theorem}

\newtheorem*{proposition*}{Proposition}
\newtheorem{lemma}[theorem]{Lemma}
\newtheorem*{lemma*}{Lemma}

\newtheorem*{conjecture*}{Conjecture}

\newtheorem*{fact*}{Fact}

\newtheorem*{exercise*}{Exercise}

\newtheorem*{hypothesis*}{Hypothesis}

\theoremstyle{definition}

\newtheorem{exercise-easy}[theorem]{Exercise}
\newtheorem{exercise-med}[theorem]{Exercise}
\newtheorem{exercise-hard}[theorem]{Exercise$^\star$}
\newtheorem{claim}[theorem]{Claim}
\newtheorem*{claim*}{Claim}

\newtheorem*{remark*}{Remark}

\newtheorem*{observation*}{Observation}

\usepackage[
letterpaper,
top=0.7in,
bottom=0.9in,
left=1in,
right=1in]{geometry}

\ifnum\arxivmode=0
\usepackage{newpxtext} 
\usepackage{textcomp} 
\usepackage[varg,bigdelims]{newpxmath}
\usepackage[scr=rsfso]{mathalfa}
\usepackage{bm} 
\let\mathbb\varmathbb
\fi

\ifnum\arxivmode=1
\usepackage[varg]{pxfonts} 
\usepackage{textcomp} 
\usepackage[scr=rsfso]{mathalfa}
\usepackage{bm} 
\fi

\ifnum\showkeys=1
\usepackage[color]{showkeys}
\fi

\makeatletter
\@ifpackageloaded{hyperref}
{}
{\usepackage[pagebackref]{hyperref}}

\definecolor{bleudefrance}{rgb}{0.01, 0.1, 1.0}
\definecolor{azure}{rgb}{0.0, 0.5, 1.0}

\ifnum\showcolorlinks=1
\hypersetup{
pagebackref=true,
colorlinks=true,
urlcolor=blue,
linkcolor=blue,
citecolor=OliveGreen}
\fi

\ifnum\showcolorlinks=0
\hypersetup{
pagebackref=true,
colorlinks=false,
pdfborder={0 0 0}}
\fi

\usepackage{prettyref}

\newcommand{\savehyperref}[2]{\texorpdfstring{\hyperref[#1]{#2}}{#2}}

\newrefformat{eq}{\savehyperref{#1}{\textup{(\ref*{#1})}}}
\newrefformat{lem}{\savehyperref{#1}{Lemma~\ref*{#1}}}
\newrefformat{def}{\savehyperref{#1}{Definition~\ref*{#1}}}
\newrefformat{thm}{\savehyperref{#1}{Theorem~\ref*{#1}}}
\newrefformat{cor}{\savehyperref{#1}{Corollary~\ref*{#1}}}
\newrefformat{cha}{\savehyperref{#1}{Chapter~\ref*{#1}}}
\newrefformat{sec}{\savehyperref{#1}{Section~\ref*{#1}}}
\newrefformat{app}{\savehyperref{#1}{Appendix~\ref*{#1}}}
\newrefformat{tab}{\savehyperref{#1}{Table~\ref*{#1}}}
\newrefformat{fig}{\savehyperref{#1}{Figure~\ref*{#1}}}
\newrefformat{hyp}{\savehyperref{#1}{Hypothesis~\ref*{#1}}}
\newrefformat{alg}{\savehyperref{#1}{Algorithm~\ref*{#1}}}
\newrefformat{rem}{\savehyperref{#1}{Remark~\ref*{#1}}}
\newrefformat{item}{\savehyperref{#1}{Item~\ref*{#1}}}
\newrefformat{step}{\savehyperref{#1}{step~\ref*{#1}}}
\newrefformat{conj}{\savehyperref{#1}{Conjecture~\ref*{#1}}}
\newrefformat{fact}{\savehyperref{#1}{Fact~\ref*{#1}}}
\newrefformat{prop}{\savehyperref{#1}{Proposition~\ref*{#1}}}
\newrefformat{prob}{\savehyperref{#1}{Problem~\ref*{#1}}}
\newrefformat{claim}{\savehyperref{#1}{Claim~\ref*{#1}}}
\newrefformat{relax}{\savehyperref{#1}{Relaxation~\ref*{#1}}}
\newrefformat{red}{\savehyperref{#1}{Reduction~\ref*{#1}}}
\newrefformat{part}{\savehyperref{#1}{Part~\ref*{#1}}}
\newrefformat{ex}{\savehyperref{#1}{Exercise~\ref*{#1}}}
\newrefformat{property}{\savehyperref{#1}{Property~\ref*{#1}}}

\newcommand{\Sref}[1]{\hyperref[#1]{\S\ref*{#1}}}

\usepackage{nicefrac}

\ifnum\fastmode=0
\ifnum\usemicrotype=1
\usepackage{microtype}
\fi
\fi

\ifnum\showauthornotes=2
\newcommand{\mynotes}[1]{{\sffamily\small\color{teal}{#1}}\medskip}
\newcommand{\Authornote}[2]{{\sffamily\small\color{blue}{[#1: #2]}}\medskip}
\newcommand{\Authornotecolored}[3]{{\sffamily\small\color{#1}{[#2: #3]}}}
\newcommand{\Authorcomment}[2]{{\sffamily\small\color{gray}{[#1: #2]}}}
\newcommand{\Authorstartcomment}[1]{\sffamily\small\color{gray}[#1: }

\newcommand{\Authorfnote}[2]{\footnote{\color{red}{#1: #2}}}
\newcommand{\Authorfixme}[1]{\Authornote{#1}{\textbf{??}}}
\newcommand{\Authormarginmark}[1]{\marginpar{\textcolor{red}{\fbox{\Large #1:!}}}}
\newcommand{\myexplain}[1]{{\sffamily\small\color{red}{\noindent [Explanation:\medskip\newline \begin{quote}#1\hfill]\end{quote}}}\medskip}

\else

\newcommand{\mynotes}[1]{}
\newcommand{\Authornote}[2]{}
\newcommand{\Authornotecolored}[3]{}
\newcommand{\Authorcomment}[2]{}
\newcommand{\Authorstartcomment}[1]{}

\newcommand{\Authorfnote}[2]{}
\newcommand{\Authorfixme}[1]{}
\newcommand{\Authormarginmark}[1]{}
\newcommand{\myexplain}[1]{}

\ifnum\showauthornotes=1
\renewcommand{\myexplain}[1]{{\sffamily\small\color{red}{\noindent \begin{quote}{\bf Explanation:} \medskip\newline #1\end{quote}}}\medskip}
\fi

\fi

\ifnum\showfixme=0

\fi

\usepackage{boxedminipage}

\newcommand{\E}{\mathbb{E}}
\newcommand{\Psymb}{\mathbb{P}}

\DeclareMathOperator*{\ProbOp}{\Psymb}

\renewcommand{\Pr}{\ProbOp}

\newcommand{\textparen}[1]{\text{(#1)}}

\ifx\because\undefined
\newcommand{\because}[1]{\textparen{because #1}}
\else
\renewcommand{\because}[1]{\textparen{because #1}}
\fi

\newcommand{\seteq}{\mathrel{\mathop:}=}

\newcommand\bdot\bullet

\ifx\mathds\undefined 
\newcommand{\Ind}{\mathbb I}
\else
\newcommand{\Ind}{\mathds 1}
\fi

\DeclareMathOperator{\sign}{sign}

\newcommand{\R}{\mathbb R}

\newcommand{\cF}{\mathcal F}

\renewcommand{\leq}{\leqslant}

\renewcommand{\geq}{\geqslant}

\ifnum\showdraftbox=1

\else

\fi

\let\epsilon=\varepsilon

\numberwithin{equation}{section}

\newcommand\MYcurrentlabel{xxx}

\newcommand{\MYstore}[2]{%
  \global\expandafter \def \csname MYMEMORY #1 \endcsname{#2}%
}

\newcommand{\MYload}[1]{%
  \csname MYMEMORY #1 \endcsname%
}

\newcommand{\MYnewlabel}[1]{%
  \renewcommand\MYcurrentlabel{#1}%
  \MYoldlabel{#1}%
}

\newcommand{\MYdummylabel}[1]{}

\newcommand{\torestate}[1]{%
  \let\MYoldlabel\label%
  \let\label\MYnewlabel%
  #1%
  \MYstore{\MYcurrentlabel}{#1}%
  \let\label\MYoldlabel%
}

\newcommand{\restatetheorem}[1]{%
  \let\MYoldlabel\label
  \let\label\MYdummylabel
  \begin{theorem*}[Restatement of \prettyref{#1}]
    \MYload{#1}
  \end{theorem*}
  \let\label\MYoldlabel
}

\newcommand{\restatelemma}[1]{%
  \let\MYoldlabel\label
  \let\label\MYdummylabel
  \begin{lemma*}[Restatement of \prettyref{#1}]
    \MYload{#1}
  \end{lemma*}
  \let\label\MYoldlabel
}

\newcommand{\restateprop}[1]{%
  \let\MYoldlabel\label
  \let\label\MYdummylabel
  \begin{proposition*}[Restatement of \prettyref{#1}]
    \MYload{#1}
  \end{proposition*}
  \let\label\MYoldlabel
}

\newcommand{\restatefact}[1]{%
  \let\MYoldlabel\label
  \let\label\MYdummylabel
  \begin{fact*}[Restatement of \prettyref{#1}]
    \MYload{#1}
  \end{fact*}
  \let\label\MYoldlabel
}

\newcommand{\restate}[1]{%
  \let\MYoldlabel\label
  \let\label\MYdummylabel
  \MYload{#1}
  \let\label\MYoldlabel
}

\newcommand{\addreferencesection}{
  \phantomsection
\ifnum\stocmode=0
  \addcontentsline{toc}{section}{References}
\else
  \addcontentsline{toc}{section}{References \hspace*{1in} --------- End of extended abstract ---------}
\fi

}

\newcommand{\e}{\epsilon}

\let\origparagraph\paragraph
\renewcommand{\paragraph}[1]{\vspace*{-10pt}\origparagraph{#1.}}

\allowdisplaybreaks

\usepackage{paralist}

\usepackage{comment}

\let\pref=\prettyref

\newcommand{\dmid}{\,\|\,}
\newcommand{\D}{\mathbb{D}}

\renewcommand{\Ind}{\vvmathbb{1}}
\ifnum\jamesmode=0
\fi
\let\1\Ind

\newcommand\f{\varphi}
\newcommand{\F}{\mathcal F}

\begin{document}

\title{Covering the large spectrum and \\ generalized Riesz products}

\author{James R. Lee\footnote{University of Washington}}

\date{}

\maketitle

\begin{abstract}
Chang's Lemma is a widely employed result in additive combinatorics.
It gives bounds on the dimension of the large spectrum
of probability distributions on finite abelian groups.
Recently, Bloom (2016) presented a powerful variant of Chang's Lemma
that yields the strongest known quantitative version of Roth's theorem
on 3-term arithmetic progressions in dense subsets of the integers.
In this note, we show how such theorems
can be derived from
the approximation of probability measures
via entropy maximization.
\end{abstract}

\section{Introduction}

Let $G$ be a finite abelian group.  Chang's Lemma \cite{Chang02} asserts
that, for every large subset $S \subseteq G$,
the large Fourier coefficients of the indicator function $\mathbf{1}_S$
lie in a low-dimensional subspace.
This has seen a number of applications in additive combinatorics
(in addition to Chang's original application to Freiman's theorem).

A theorem of Bloom \cite{Bloom16} shows that a large subset of the large spectrum
can be contained in an even lower-dimensional subspace.
We refer to \pref{sec:large-spectrum} for the formal statements.
Bloom employs his theorem as the key tool in obtaining the following
quantitative version of Roth's theorem.

\begin{theorem}
There exists a $c > 0$ such that
for all sufficiently large $N$, the following holds:
If $A \subseteq \{1,\ldots,N\}$ contains no non-trivial three-term
arithmetic progression, then
\[
|A| \leq c \frac{(\log \log N)^4}{\log N} N\,.
\]
\end{theorem}
This improves slightly over Sanders' \cite{Sanders11} breakthrough result
that has $(\log \log N)^4$ replaced by $(\log \log N)^6$.

In this note, we state a general approximation theorem
for probability measures on finite spaces equipped
with no algebraic structure.
From this theorem, Bloom's result follows easily.
While Bloom's proof uses the additive structure
in a seemingly fundamental and intricate way,
our argument
is elementary and requires only a direct application
of the fact that the characters
of a finite abelian group
are homomorphisms and bounded in $\ell_{\infty}$.

The statement and proof
are inspired by the ``entropy maximization'' philosophy:
Given a probability measure $\mu$ and a collection
of linear observables $\ccF$, one can find a ``simple'' approximator $\tilde \mu$ (with respect
to $\ccF$) by maximizing the entropy of $\tilde \mu$ over
all probability measures having similar behavior on $\ccF$.

Our use of this philosophy is motivated by the work \cite{LRS15}
where it is employed in the setting of quantum states and von Neumann entropy.
In \cite{IMR14}, the authors use a simple entropy argument
to prove the special case of Chang's Lemma when $G=\mathbb{F}_2^n$.
The entropy-maximization approach
is also related, at least in spirit,
to the works \cite{Gowers10} and \cite{RTTV08} on ``dense model theorems,''
and to a long line of works employing an ``entropy regularizer''
in the setting of convex optimization.
For a discussion of these connections,
additional applications of our sparse approximation
theorem, and further accounts of the use
of relative entropy in additive combinatorics,
we refer to the forthcoming paper of Wolf \cite{Wolf17}.

In the next section, we state and prove
an approximation theorem
in the context of finite probability spaces.
In \pref{sec:large-spectrum}, we prove the results
of Bloom and Chang.


\section{An approximation theorem}

Let $X$ be a finite set equipped with a probability measure $\mu$.
We use $L^2(\mu)$ to denote the Hilbert space of real-valued functions on $X$
equipped with inner product $\langle f,g\rangle = \sum_{x \in X} \mu(x) f(x) g(x)$.
For a function $h : X \to \R$, we will use the notation
$\E_{\mu} h = \sum_{x \in X} \mu(x) h(x)$.
We also denote by $\|h\|_p=(\E_{\mu} |h|^p)^{1/p}$ the $L^p(\mu)$ norm for $p \geq 1$.

 Denote the set of densities with respect to $\mu$ by
$\Delta_X = \{ f : X \to [0,\infty) : \|f\|_1=1\}$.
For $f \in \Delta_X$, define the relative entropy
\[
\Ent_{\mu}(f) = \mathbb E_{\mu} [f \log f]\,.
\]
We will also use the notion of the relative entropy
between two densities $h,h' \in \Delta_X$:
\[
\D_{\mu}(h\dmid h') = \E_{\mu} \left[h \log \frac{h}{h'}\right]\,.
\]
This definition makes sense whenever $\mathrm{supp}(h) \subseteq \mathrm{supp}(h')$.
Otherwise, we take the value to be $+\infty$.

\paragraph{Generalized Riesz products}

Suppose that $\mathcal F \subseteq L^{2}(\mu)$ is a collection
 satisfying $\sup_{\varphi \in \mathcal F} \|\varphi\|_{\infty} \leq 1$.
 Define the semi-norm $\|f\|_{\mathcal F} = \sup_{\varphi \in \mathcal F} |\langle \varphi, f\rangle|$.
Say that a function $R \in L^2(\mu)$ is a {\em degree-$d$ Riesz $\ccF$-product} if
\[
R(x) = \prod_{i=1}^d (1+\e_i \varphi_i(x))
\]
for some $d \geq 1$ and $\varphi_1, \ldots, \varphi_d \in \mathcal F$, $\e_1, \ldots, \e_d \in \{-1,0,1\}$.
Observe that every such $R$ is non-negative on $X$.


\begin{theorem}[Sparse approximation theorem]\label{thm:approx}
For every $0 < \eta < \frac{1}{e^3}$ and $f \in \Delta_X$,
there is a $g \in \Delta_X$ such that:
\begin{enumerate}
\item $\|f-g\|_{\ccF} \leq \eta\,.$
\item There is a subset $\cF' \subseteq \cF$ with
   \begin{equation}\label{eq:junta-size}
      |\cF'| \leq 9 \frac{\Ent_{\mu}(f)}{\eta^2}\,,
   \end{equation}
   and such that $g$ is a non-negative linear combination of degree-$d$ Riesz $\cF'$-products for
   \begin{equation}
   d \leq 12 \frac{\Ent_{\mu}(f)}{\eta} + O\left(\frac{\log \frac{1}{\eta}}{\log \log \frac{1}{\eta}}\right) \label{eq:deg}
   \end{equation}
\end{enumerate}
\end{theorem}

While \pref{thm:approx} yields a result that is closely related to Chang's Lemma
and is sufficient for the case $G=\mathbb{F}_2^n$, it seems that a more delicate property
is required to recover the full statement.
Say that the family $\cF$ is {\em Laplace pseudorandom} if 
for every collection $\{ \lambda_{\f} : \f \in \cF\}$ of real numbers,
the following property holds:
\begin{equation}\label{eq:laplace}
   \log \E_{\mu} \left[\exp\left(\sum_{\f \in \cF} \lambda_{\f} \f\right)\right]
   \leq \frac12 \sum_{\f \in \cF} \lambda_{\f}^2\,.
\end{equation}

\begin{lemma}\label{lem:laplace}
   If $\cF$ is Laplace pseudorandom then for any $f \in \Delta_X$, it holds that
   \[
      \sum_{\f \in \cF} \langle f, \f\rangle^2 \leq 2\, \Ent_{\mu}(f)\,.
   \]
\end{lemma}

\subsection{Duality theory for relative entropy minimization}

\pref{lem:laplace} and part of \pref{thm:approx} can be proved
using only elementary properties of duality for optimization
of convex functions over polytopes.
Establishing the bound \eqref{eq:junta-size} will require
an iterative algorithm described in \pref{sec:proofs}.

Fix some $f \in \Delta_X$, a finite collection $\cF_0 \subseteq L^2(\mu)$, and
a parameter $\delta \geq 0$.
Consider the optimization:
\begin{align}
   \textrm{mininize}& \quad \Ent_{\mu}(g) \label{eq:primal}\\
   \textrm{subject to}& \quad g \in \Delta_X \nonumber\\
                      &\quad \langle g,\f\rangle \geq \langle f,\f\rangle - \delta\quad \forall \f \in \cF_0\,.\nonumber
\end{align}
Note that we are minimizing a strongly convex function over a non-empty, compact polytope (since $f$ itself
satisfies all the constraints), and thus \eqref{eq:primal} has a unique optimal solution.
The corresponding dual optimization is
\begin{align}\label{eq:dual}
   \textrm{maximize}& \quad 
   - \log \left(\E_{\mu} \exp\left(\sum_{\f \in \cF_0} \lambda_{\f} \f\right)\right)
   + \sum_{\f \in \cF_0} \lambda_\f \left(\langle f,\f\rangle-\delta\right)
   \\
   \vphantom{\frac{\bigoplus}{\bigoplus}}
   \textrm{subject to}& \quad \lambda_{\f} \geq 0 \qquad \forall \f \in \cF_0\,.\nonumber
\end{align}
See, for instance, \cite[\S 5.2.4]{BV04}.

Let $\primo$ and $\duo$ denote the optimal values of \eqref{eq:primal} and \eqref{eq:dual}, respectively.
By weak duality, the inequality $\primo \geq \duo$ always holds.
Let us use this fact to prove \pref{lem:laplace}.

\begin{proof}[Proof of \pref{lem:laplace}]
   Consider the optimizations \eqref{eq:primal} and \eqref{eq:dual} with $\delta=0$ and
   \[
      \cF_0 = \left\{ \sign(\langle f,\f\rangle)\, \f : \f \in \cF \right\}
   \]
   so that $\langle f,\f\rangle \geq 0$ for $\f \in \cF_0$.
   Then by weak duality:
   \begin{align*}
      \Ent_{\mu}(f) \geq \primo \geq \duo 
      &\geq
      - \log \left(\E_{\mu} \exp\left(\sum_{\f \in \cF_0} \langle f,\f\rangle \,\f\right)\right) 
      + \sum_{\f \in \cF_0} \langle f,\f\rangle^2\,,
   \end{align*}
   where the last inequality employs the feasible solution $\{\lambda_{\f} = \langle f,\f\rangle : \f \in \cF_0\}$.

   Using the assumption that $\cF$ is Laplace pseudorandom, this yields
   \[
      \Ent_{\mu}(f) \geq \frac12 \sum_{\f \in \cF} \langle f,\f\rangle^2\,,
   \]
   completing the proof.
\end{proof}

For $\delta > 0$,
the optimization \eqref{eq:primal} is strictly feasible since (as witnessed by $f$),
and hence Slater's theorem implies that strong duality holds and $\primo = \duo$
(see, e.g., \cite[\S 5.3.2]{BV04}).
In this case, the KKT conditions hold, i.e., the gradient of the Lagrangian
is identically zero at the optimal solution.

Let $(g^*, \{\lambda_{\f}^*\})$ denote the corresponding optimal primal-dual pair.
The gradient condition yields
\begin{equation}\label{eq:primal-opt}
   g^* = \frac{\exp\left(\sum_{\f \in \cF_0} \lambda^*_{\f} \f\right)}{\E_{\mu} \exp\left(\sum_{\f \in \cF_0} \lambda^*_{\f} \f\right)}\,.
\end{equation}
It follows that
\begin{equation}\label{eq:dual-one}
   \Ent_{\mu}(f) - \D_{\mu}(f \dmid g^*) = \mathbb{E}_{\mu}\left[f \log g^*\right]
   = \duo + \delta \sum_{\f \in \cF_0} \lambda^*_{\f}\,,
\end{equation}
where the latter equality uses $\E_{\mu} f = 1$.

\begin{lemma}\label{lem:dual-bound}
   For every $\delta > 0$, the optimal solution $\{\lambda_{\f}^*\}$ of \eqref{eq:dual} satisfies
   \[
      \sum_{\f \in \cF_0} \lambda^*_{\f} \leq \frac{\Ent_{\mu}(f)}{\delta}\,.
   \]
\end{lemma}

\begin{proof}
   Note that $\duo \geq 0$ because $\lambda_{\f} \equiv 0$ is a feasible solution.
   Therefore \eqref{eq:dual-one} yields
   \[
      \delta \sum_{\f \in \cF_0} \lambda^*_{\f} \leq \Ent_{\mu}(f) - D_{\mu}(f \dmid g^*) \leq \Ent_{\mu}(f)\,.\qedhere
   \]
\end{proof}

\subsection{Truncating the exponential}
\label{sec:truncate}

Let us now move on to the proof of \pref{thm:approx}.

\begin{lemma}\label{lem:truncate}
   Suppose that $\|\f\|_{\infty} \leq 1$ for $\f \in \cF_0 \subseteq L^2(\mu)$.
   Consider non-negative numbers $\{c_{\f} : \f \in \cF_0\}$ and
   \[
      h = \exp\left(\sum_{\f \in \cF_0} c_{\f} (1+\f)\right)\,.
   \]
   Then for every $0 < \eta < \frac{1}{e^3}$,
   there is a density $\tilde h \in \Delta_X$ that is a non-negative
   linear combination of degree-$d$ Riesz $\cF_0$-products and such that
   \[
      d \leq 6 \sum_{\f \in \cF_0} c_{\f} + O\left(\frac{\log \frac{1}{\eta}}{\log \log \frac{1}{\eta}}\right)\,,
   \]
   and
   \[
      \left\|\frac{h}{\E_{\mu} h}-\tilde{h}\right\|_1 \leq \eta\,.
   \]
\end{lemma}

\begin{proof}
   Let $\psi = \sum_{\f \in \cF_0} c_{\f} (1+\f)$ and note that 
   each summand is non-negative (because $\|\f\|_{\infty} \leq 1$) and
   $\|\psi\|_{\infty} \leq 2 c$, where
   $c = \sum_{\f \in \cF_0} c_{\f}$.
Denote $p_m(x) = \sum_{j \leq m} \frac{x^j}{j!}$ and recall from Taylor's thoerem
that for $B \geq 0$,
\[
\sup_{x \in [0,B]} \frac{|e^x - p_m(x)|}{e^x} \leq \frac{B^{m+1}}{(m+1)!}\,.
\]
Let us choose $m \leq 3B + O\left(\frac{\log \frac{1}{\eta}}{\log \log \frac{1}{\eta}}\right)$
so as to make this quantity less than $\eta/2$.
Thus setting $B=2c$
yields
\begin{equation}\label{eq:approx1}
   \|e^{\psi}-p_m(\psi)\|_1 \leq \frac{\eta}{2} \E_{\mu} [e^{\psi}]\,.
\end{equation}

Now define
\[
   \tilde{h} = \frac{p_m(\psi)}{\E_{\mu} p_m(\psi)}\,,
\]
and note that $h$ is a non-negative combination of degree-$m$ Riesz $\cF_0$-products.
Moreover,
\[
   \left\|\frac{h}{\E_{\mu} h} - \tilde h\right\|_1 \leq \left\|\frac{h}{\E_{\mu} h} - \frac{p_m(\psi)}{\E_{\mu} h}\right\|_1 +
   \left\|\frac{p_m(\psi)}{\E_{\mu} h}-\tilde{h}\right\|_1 
   \stackrel{\eqref{eq:approx1}}{\leq} \frac{\eta}{2}
   + \left|\frac{\E_{\mu} p_m(\psi)}{\E_{\mu} h} - 1\right| \stackrel{\eqref{eq:approx1}}{\leq} \eta\,.\qedhere
\]
\end{proof}

We first prove \pref{thm:approx} without the sparsity constraint \eqref{eq:junta-size} since
it follows easily from the machinery we already have.

\begin{theorem}[Low-degree approximation theorem]
   \label{thm:approx2}
For every $0 < \eta < \frac{1}{e^3}$ and $f \in \Delta_X$,
there is a $g \in \Delta_X$ such that:
\begin{enumerate}
\item $\|f-g\|_{\ccF} \leq \eta\,.$
\item $g$ is a non-negative linear combination of degree-$d$ Riesz $\cF$-products for
   \begin{equation}
   d \leq 12 \frac{\Ent_{\mu}(f)}{\eta} + O\left(\frac{\log \frac{1}{\eta}}{\log \log \frac{1}{\eta}}\right)\,.
   \end{equation}
\end{enumerate}
\end{theorem}

\begin{proof}
   Consider the optimization \eqref{eq:primal} with $\delta = \eta/2$ and $\cF_0 = \{ \pm \f : \f \in \F \}$.
   Let $(g^*, \{\lambda_{\f}^*\})$ denote the corresponding optimal primal-dual pair and observe that
   \[
      g^* = \frac{\exp\left(\sum_{\f \in \cF_0} \lambda^*_{\f} (1+\f)\right)}{\E_{\mu}\exp\left(\sum_{\f \in \cF_0} \lambda^*_{\f} (1+\f)\right)}\,.
   \]
   Moreover, \pref{lem:dual-bound} asserts that $c = \sum_{\f \in \cF_0} \lambda^*_{\f} \leq 2 \frac{\Ent_{\mu}(f)}{\eta}$.

   Thus we can apply \pref{lem:truncate} to obtain a density $\tilde{h} \in \Delta_X$ that is a non-negative
   linear combination of degree-$d$ Riesz $\cF_0$-products with
   \[
      d \leq 12 \frac{\Ent_{\mu}(f)}{\eta} +O\left(\frac{\log \frac{1}{\eta}}{\log \log \frac{1}{\eta}}\right)\,,
   \]
   and such that $\|\tilde{h}-g^*\|_1 \leq \eta/2$.

   Finally, observe that for any $\f \in \cF$, by definition of the optimization \eqref{eq:primal}, we have
   \[
      |\langle \tilde{h}-f,\f\rangle| \leq |\langle \tilde{h}-g^*,\f\rangle| + |\langle g^*-f,\f\rangle| 
      \leq \|\tilde{h}-g^*\|_1 + \frac{\eta}{2} \leq \eta\,,
   \]
   where in the second inequality we have used $\|\f\|_{\infty} \leq 1$.  It follows that $\|\tilde{h}-f\|_{\cF} \leq \eta$,
   completing the proof.
\end{proof}

\subsection{Mirror descent}
\label{sec:proofs}

We now prove \pref{thm:approx} by giving an algorithm
that approximately solves the optimization \eqref{eq:primal}.
The algorithm and analysis are
based on the ``mirror descent'' framework, analyzed
using a Bregman divergence (in this case, the relative entropy).  See, for instance,
the monograph \cite{Bubeck2014}.
The sparsity of the solution (captured by \eqref{eq:junta-size})
is closely related to sparsity properties of the Frank-Wolfe algorithm \cite{FW56}.

Assume that $\eta > 0$ and $f \in \Delta_X$
are given as in the theorem.
For some value $T > 0$,
define a family $\{g_t : t \in [0,T] \} \subseteq \Delta_X$ by
\begin{equation}\label{eq:gtform}
g_t = \frac{\exp\left(\int_0^t \varphi_s\,ds\right)}{\E_{\mu} \exp\left(\int_0^t \varphi_s\,ds\right)}\,,
\end{equation}
where $s \mapsto \varphi_s \in L^2(\mu)$ is a measurable function
to be specified shortly.  Observe that $g_0 = \mathbf{1}$ is the constant $1$ function.

A simple calculation yields:  For $t \in [0,T)$,
\begin{equation}\label{eq:deriv}
\frac{d}{dt} \D_{\mu}(f\dmid g_t) = \langle \varphi_t, g_t-f\rangle\,.
\end{equation}

We define the maps $s \mapsto \varphi_s$ to be piecewise constant on
a finite sequence of intervals.  Given the definition on
intervals $[0,t_1),[t_1,t_2), \ldots,[t_{i-1},t_i)$ with $0 < t_1 < t_2 < \cdots < t_i$,
we define it on an interval $[t_i,t_{i+1})$ as follows.

If there exists a functional $\varphi \in \cF$ such that
\[
|\langle g_{t_i}, \varphi\rangle - \langle f,\varphi\rangle| > \frac{2\eta}{3}\,,
\]
then we put
\begin{equation}\label{eq:phis}
\varphi_s = \mathrm{sign} \left(\langle f-g_{t_i},\varphi\rangle\right) \cdot \varphi
\end{equation}
for $s \in [t_i, t_{i+1})$
where $t_{i+1} = \inf \{ t \geq t_i : |\langle g_t,\varphi\rangle -\langle f,\varphi\rangle| \leq \eta/3\}$.
We will see momentarily why such a $t_{i+1}$ must exist.

If there is no such functional $\varphi$ at time $t_i$, then we set $T=t_i$ and $i_{\max}=i$.
By construction, we have the property that $\|f-g_T\|_{\cF} \leq \frac23 \eta$.

\begin{lemma}\label{lem:T}
$T \leq 3\frac{\Ent_{\mu}(f)}{\eta}\,.$
\end{lemma}

\begin{proof}
Simply observe that for $t \in [0,T)$, the calculation \eqref{eq:deriv} combined
with the definition of the sequence $\{t_i\}$ and the choice \eqref{eq:phis} yields
\[
\frac{d}{dt} \D_{\mu}(f\dmid g_t) \leq -\frac{\eta}{3}\,.
\]
On the other hand, $\D_{\mu}(f\dmid g_0) = \Ent_{\mu}(f)$ and $\D_{\mu}(f \dmid g_t) \geq 0$ is always true.
This yields the claim.
\end{proof}

\begin{lemma}\label{lem:i}
It holds that $i_{\max} \leq 9\frac{\Ent_{\mu}(f)}{\eta^2}$.
\end{lemma}

\begin{proof}
Fix an interval $[t_{i-1}, t_i)$ with $i \leq i_{\max}$.
Let $\varphi = \varphi_{t_{i-1}}$.  We calculate
\[
\frac{d}{dt} \langle \varphi,g_t\rangle = - \langle \varphi, g_t(\varphi-\langle \varphi,g_t\rangle)\rangle
= - \langle \varphi^2,g_t\rangle + \langle \varphi,g_t\rangle^2\,.
\]
Notice that the latter quantity is at least $-\|\varphi\|_{\infty}^2 \|g_t\|_1 \geq -1$.
Therefore $t_i-t_{i-1} \geq \frac{\eta}{3}$.  We conclude
that $i_{\max} \leq 3T/\eta$ and combine this with \pref{lem:T}.
\end{proof}

Observe
now that
\begin{equation}\label{eq:gT}
   g_T = \frac{\exp\left(\int_0^T (1+\varphi_s)\,ds\right)}{\E_{\mu} \exp\left(\int_0^T (1+\varphi_s)\,ds\right)}
\end{equation}
and $\|f-g_T\|_{\cF} \leq 2\eta/3$.

Note that if we set
\[
   \cF' = \left\{ \f \in \cF : \f = \pm \f_t \textrm{ for some } t \in [0,T] \right\}\,,
\]
then \pref{lem:i} yields $|\cF'| \leq 9 \frac{\Ent_{\mu}(f)}{\eta}$.
The proof of \pref{thm:approx}
is concluded using \pref{lem:truncate} in conjunction with \pref{lem:T},
just as in the proof of \pref{thm:approx2}.

\section{Covering the large spectrum}
\label{sec:large-spectrum}

Let $G$ be a finite abelian group equipped
with the uniform measure $\mu$, and let $\hat G$ be the dual group.
Let $0$ denote the identity element in $G$ and $\hat G$.

For $\gamma \in \hat G$, let $u_{\gamma} : G \to \mathbb C$
denote the corresponding character.   One can write any $f:G \to \mathbb C$
as $f = \sum_{\gamma \in \hat G} \hat f(\gamma) u_{\gamma}$.
We will need the properties that $u_{\gamma} u_{\gamma'} = u_{\gamma+\gamma'}$ for
all $\gamma,\gamma' \in \hat G$ and
$\max_{x \in G} |u_{\gamma}(x)| \leq 1$.
One may consult \cite[Ch. 4]{TaoVu10} for a treatment
of discrete Fourier analysis
tailored to applications in additive combinatorics.

For each value $\delta > 0$, we define the set
\[
\Spec_{\delta}(f) = \{ \gamma \in \hat G : |\hat f(\gamma)| > \delta\}\,.
\]
Say that a subset $S \subseteq \hat G$ is {\em covered} by a subset $\Lambda \subseteq \hat{G}$ if
\[
S \subseteq \left\{ \sum_{\lambda \in \Lambda} \epsilon_{\lambda} \lambda : \e_{\lambda} \in \{-1,0,1\}\right\}\,.
\]
A subset $S \subseteq \hat G$ is
{\em $d$-covered} if there exists a subset $\Lambda \subseteq \hat G$
with $|\Lambda|\leq d$ that covers $S$.

Let us define the family
\[
\ccF = \left\{ \mathrm{Re}\,u_{\gamma}, \mathrm{Im}\,u_{\gamma} : \gamma \in \hat G \right\} \subseteq L^2(\mu)\,.
\]
Note that $\|\varphi\|_{\infty} \leq 1$ for every $\varphi \in \mathcal F$.

\begin{lemma}\label{lem:basic}
If $R$ is a degree-$d$ Riesz $\ccF$-product, then $\Spec_0(R) = \{\gamma \in \hat G : \hat R(\gamma) \neq 0\}$ is $d$-covered.
\end{lemma}

\begin{proof}
Write $R=\prod_{i=1}^d (1+\e_i \varphi_i)$ for $\{\varphi_i\} \subseteq \mathcal F$ and $\{\e_i\} \subseteq \{-1,0,1\}$.
For each $i$, let $\gamma_i \in \hat G$ be such that $\varphi_i = \Re u_{\gamma_i}$ or $\varphi_i = \Im u_{\gamma_i}$.
Since we can write $\mathrm{Re}\,u_{\gamma} = \frac12(u_{\gamma}+u_{-\gamma})$ and
$\mathrm{Im}\,u_{\gamma}=\frac{1}{2i}(u_{\gamma}-u_{-\gamma})$,
upon expanding the product defining $R$, we see that every
$\gamma \in \hat G$ with $\hat R(\gamma) \neq 0$ is a sum of
at most $d$ elements from the {\em multiset} $\Gamma_0 \seteq \{\gamma_1, \ldots, \gamma_d, -\gamma_1, \ldots, -\gamma_d\} \subseteq \hat G$.
(We are using the convention here that the empty sum is equal to the identity of $\hat G$
in order to handle $\hat R(0) \neq 0$.)
But we can replace $\Gamma_0$ by an actual set $\Gamma \subseteq \hat{G}$ as follows:
For each $i=1,\ldots,d$, 
if $\gamma_i$ occurs $t$ times in $\Gamma_0$, we replace the $t$ occurrences of $\pm \gamma_i$ by the elements
$\{\pm \gamma_i, \pm 2 \gamma_i, \cdots, \pm t \gamma_i\}$.
\end{proof}

\subsection{Bloom's theorem}


Recall that $\Delta_G = \left\{ f : G \to [0,\infty) : \E_{\mu} f = 1\right\}$.


\begin{theorem}[Bloom]
   \label{thm:bloom}
For every $f \in \Delta_G$ and $0 < \delta < \frac{1}{e^3}$, there exists
a subset $S \subseteq \Spec_{\delta}(f)$ such that $|S| \geq \frac{\delta}{2} |\Spec_{\delta}(f)|$ and $S$ is $d$-covered
for
\[
d \leq 24\sqrt{2} \frac{\Ent_{\mu}(f)}{\delta} + O\left(\frac{\log \frac{1}{\delta}}{\log \log \frac{1}{\delta}}\right)\,.
\]
\end{theorem}

\begin{proof}
Setting $\eta = \delta/(2\sqrt{2})$ and applying
\pref{thm:approx}, there exists a $g \in \Delta_G$ such that
\[
g = \sum_{i=1}^N c_i R_i
\]
with $N \geq 1$, $c_1, \ldots, c_N > 0$, and where $R_1, \ldots, R_N$
are degree-$d$ Riesz $\ccF$-products for $d$ as in \eqref{eq:deg}
and furthermore $\|f-g\|_{\mathcal F} \leq \eta\,.$

\medskip

Observe that since $g \in \Delta_G$, we have $\sum_{i=1}^N c_i \mathbb E_{\mu} R_i = \E_{\mu} g = 1$.
Thus we can define a random variable $Z \in \{1,2,\ldots,N\}$ so that
\[
\Pr[Z=i] = c_i \mathbb E_{\mu} R_i\,.
\]

Since $\|f-g\|_{\cF} \leq \eta$, we deduce that
if $\gamma \in \Spec_{2 \sqrt{2} \eta}(f)$, then $\gamma \in \Spec_{\sqrt{2} \eta}(g)$.
For such $\gamma$, we have
\[
\E_z\left[\left|\langle u_{\gamma}, \tfrac{R_z}{\E_{\mu} R_z}\rangle\right|\right] =\sum_{i=1}^N c_i (\mathbb E_{\mu} R_i) \left|\left\langle u_{\gamma}, \tfrac{R_i}{\E_{\mu} R_i}\right\rangle\right|
\geq |\langle u_{\gamma},g\rangle| \geq \sqrt{2}\eta = \frac{\delta}{2}\,.
\]

Because $\left|\left\langle u_{\gamma}, \tfrac{R_i}{\E_{\mu} R_i}\right\rangle\right| \leq 1$,
we conclude that
\begin{equation*}
\label{eq:yup}
\Pr_z\left(\hat R_z(\gamma) \neq 0\right) = \Pr_z\left(|\langle u_{\gamma}, R_z\rangle| > 0\right) \geq \frac{\delta}{2}\,.
\end{equation*}
By linearity, $\E_z |\Spec_0(R_z)|  \geq \frac{\delta}{2} \left|\Spec_{\delta}(f)\right|.$
Moreover, by \pref{lem:basic}, every set $\Spec_0(R_i)$ is $d$-covered.  Thus
there exists at least one such set that completes the proof of the theorem.
\end{proof}

\subsection{Chang's theorem}

\begin{theorem}[Chang]\label{thm:chang}
   For every $f \in \Delta_G$ and $\delta > 0$, the set $\Spec_{\delta}(f)$ is $d$-covered for
   \[
      d \leq 4 \frac{\Ent_{\mu}(f)}{\delta^2}\,.
   \]
\end{theorem}

Note that \pref{thm:approx} implies 
there is a density $g \in \Delta_G$ such that $\Spec_{\delta}(f) \subseteq \Spec_0(g)$
and from \eqref{eq:junta-size}, one can write
$g(x) = \psi(u_{\gamma_1}(x), \ldots, u_{\gamma_k}(x))$ for some function $\psi$
and $\gamma_1, \ldots, \gamma_k \in \hat{G}$ with $k \leq O(\Ent_{\mu}(f)/\delta^2)$.
In the special case $G=\mathbb{F}_2^n$, this implies that
\[
   \Spec_0(g) \subseteq \mathrm{span}_{\mathbb{F}_2}(\gamma_1, \ldots, \gamma_k) = \left\{ \sum_{i=1}^k \e_i \gamma_i : \e_i \in \{-1,0,1\} \right\}\,,
\]
yielding \pref{thm:chang} for $G=\mathbb{F}_2^n$.
For general finite abelian $G$, this no longer holds, and one obtains instead the following statement.

\begin{lemma}
   For every $f \in \Delta_G$ and $0< \delta < \frac{1}{e^3}$, there is a set $\Lambda \subseteq \hat{G}$ with
   \[
      |\Lambda| \leq 18 \frac{\Ent_{\mu}(f)}{\delta^2}
   \]
   and such that every element $\gamma \in \Spec_{\delta}(f)$ can be written
   \[
      \gamma = \sum_{i=1}^d \e_i \gamma_i\,,\qquad (\gamma_1, \ldots,\gamma_d)\in \Lambda^d, \e_1, \ldots,\e_d \in \{-1,0,1\}\,.
   \]
   with
   \[
      d \leq 12\sqrt{2} \frac{\Ent_{\mu}(f)}{\delta} + O\left(\frac{\log \frac{1}{\delta}}{\log \log \frac{1}{\delta}}\right)\,.
   \]
\end{lemma}

This should be compared to \cite[Thm. 4]{Shkredov06} which achieves a worse bound on $|\Lambda|$
but the significantly better bound $d \leq O(\Ent_{\mu}(f))$.

\begin{proof}
   As in the proof of \pref{thm:bloom}, set $\eta = \delta/\sqrt{2}$ and apply \pref{thm:approx} to
   obtain a density $g = \sum_{i=1}^N c_i R_i$ where each $R_i$ is a degree-$d$ Riesz $\cF$-product
   with $d$ as in \eqref{eq:deg}.
   Now $\gamma \in \Spec_{\delta}(f)$ implies $\gamma \in \Spec_{0}(g)$, which means that $\gamma \in \Spec_0(R_i)$
   for some $i=1,\ldots,N$.

   To conclude, observe that every element of $\Spec_0(R_i)$ can be written as $\sum_{i=1}^d \e_i \gamma_i$
   for some {\em tuple} $(\gamma_1, \ldots, \gamma_d) \in \Lambda^d$ (recall the proof
   of \pref{lem:basic}).
\end{proof}

\smallskip

In order to prove \pref{thm:chang} for general $G$, we recall the following definition.
Say that a subset $\Lambda \subseteq \hat{G}$ is {\em disassociated}
if 
\[
   \sum_{\gamma \in \Lambda} \e_{\gamma} \gamma = 0 \textrm{ and } \{\e_{\gamma}\} \subseteq \{-1,0,1\} \implies
   \e_{\gamma} = 0 \quad\forall \gamma \in \Lambda\,.
\]
If $\Lambda \subseteq \Spec_{\delta}(f)$ is a {\em maximal} disassociated subset, then
$\Spec_{\delta}(f)$ is covered by $\Lambda$.  Thus the following lemma
finishes the proof of \pref{thm:chang}.
The argument is based on a a proof of Rudin's inequality
credited to I. Z. Ruzsa in \cite{Green04}.

\begin{lemma}
   If $\Lambda \subseteq \Spec_{\delta}(f)$ is disassociated, then
   \[
      |\Lambda| \leq 4 \frac{\Ent_{\mu}(f)}{\delta^2}\,.
   \]
\end{lemma}

\begin{proof}
   Let $\cF_1 = \{ \Re u_{\gamma}: \gamma \in \Lambda\}, \cF_2 = \{ \Im u_{\gamma} : \gamma \in \Lambda\}$.

   \begin{claim}\label{claim:laplace}
      The families $\cF_1$ and $\cF_2$ are Laplace pseudorandom.
   \end{claim}

   Given \pref{claim:laplace}, we have
   \[
      |\Lambda| \delta^2 \leq \sum_{\f \in \cF_1 \cup \cF_2} \langle f,\f\rangle^2 \leq 4\, \Ent_{\mu}(f)\,,
   \]
   where the first inequality follows from $\Lambda \subseteq \Spec_{\delta}(f)$ and the second is \pref{lem:laplace}.

   So let us turn to the proof of \pref{claim:laplace}.
   We prove it for $\cF_1$ as the proof for $\cF_2$ is essentially identical.
   We require the following two basic facts:  For any $t \in \R$ and $x \in [-1,1]$,
   \begin{align}
      e^{tx} &\leq \frac{e^t+e^{-t}}{2} + x \frac{e^t - e^{-t}}{2} = \cosh(t) + x \sinh(t)\label{eq:taylor1}\,, \\
      \cosh(t) &= \sum_{k \geq 0} \frac{t^{2k}}{(2k)!} \leq \sum_{k \geq 0} \frac{t^{2k}}{2^k k!} = e^{t^2/2}\,.
    \label{eq:taylor2}
   \end{align}
   The first uses the fact that $x \mapsto e^{tx}$ is convex.

   Now write
   \begin{equation}\label{eq:bernstein}
      \E_{\mu}\left[\exp\left(\sum_{\f \in \cF_1} \lambda_{\f} \f\right)\right] \stackrel{\eqref{eq:taylor1}}{\leq}
      \E_{\mu} \prod_{\f \in \cF_1} \left(\cosh(\lambda_{\f}) + \f \sinh(\lambda_{\f})\right)
   \end{equation}
   Recalling that every $\f \in \cF_1$ is of the form $\f = \Re u_{\gamma} = \frac12 (u_{\gamma} + u_{-\gamma})$ for some $\gamma \in \Lambda$, we see
   that the right-hand side of \eqref{eq:bernstein}
   breaks into a linear combination of characters $u_{\alpha}$ such that
   \[
   \alpha = \sum_{\gamma \in \Lambda} \e_{\gamma} \gamma\,, \quad \e_{\gamma} \in \{-1,0,1\}\,.
   \]
   But $\E_{\mu}[u_{\alpha}]=0$ unless $\alpha=0$.
   By disassociativity of $\Lambda$, this can only happen if $\e_\gamma = 0$ for all $\gamma \in \Lambda$.
   In particular, we conclude that
   \[
      \E_{\mu}\left[\exp\left(\sum_{\f \in \cF_1} \lambda_{\f} \f\right)\right] \leq \prod_{\f \in \cF_1} \cosh(\lambda_{\f})
      \leq \exp\left(\frac12 \sum_{\f \in \cF_1} \lambda_{\f}^2\right)\,,
   \]
   implying that $\cF_1$ is Laplace pseudorandom and completing the argument.
\end{proof}

\smallskip

\subsection*{Acknowledgements}

We thank Thomas Bloom, Prasad Raghavendra, and Julia Wolf for enlightening discussions,
and Thomas Vidick for detailed comments on an initial draft of this manuscript.
We are also grateful to Julia Wolf for pointing out the reference \cite{Shkredov06}.

\bibliographystyle{alpha}
\bibliography{chang}

\end{document}